\documentclass[12pt,a4paper]{amsart}

\usepackage{amsfonts, amsmath, amssymb, amsthm, amscd, hyperref}

\newtheorem{thm}{Theorem}
\newtheorem*{thm*}{Theorem}
\newtheorem{lem}{Lemma}

\newtheorem{cor}[thm]{Corollary}

\theoremstyle{definition}
\newtheorem{defn}{Definition}

\theoremstyle{remark}

\DeclareMathOperator{\cat}{cat} 
\DeclareMathOperator{\secat}{secat}
\DeclareMathOperator{\hl}{hlen}

\DeclareMathOperator{\cl}{cl}

\renewcommand{\epsilon}{\varepsilon}

\begin{document}

\title{The genus and the Lyusternik-Schnirelmann category of preimages}

\author{R.N.~Karasev}
\thanks{This research is supported by the Dynasty Foundation, the President's of Russian Federation grant MK-113.2010.1, the Russian Foundation for Basic Research grants 10-01-00096 and 10-01-00139}

\email{r\_n\_karasev@mail.ru}
\address{
Roman Karasev, Dept. of Mathematics, Moscow Institute of Physics
and Technology, Institutskiy per. 9, Dolgoprudny, Russia 141700}



\begin{abstract}
In this paper some axiomatic generalization (function of open subsets) of the relative Lyusternik-Schnirelmann category is considered, incorporating the sectional category and the Schwarz genus as well. For this function and a given continuous map of the underlying space to a finite-dimensional metric space some lower bounds on the value of this function on the (neighborhood of) preimage of some point are given.
\end{abstract}

\maketitle

\section{Introduction}

Let us introduce a notion that generalizes the relative Lyusternik-Schnirelmann category.

\begin{defn}
Let $X$ be a topological space. Let the function $\kappa$ takes the set of nonempty open subsets of $X$ to positive integers and has the following properties:

1) if $U\subseteq V$, then $\kappa(U) \le \kappa(V)$ (monotonicity);

2) $\kappa(U_1\cup\dots\cup U_n) \le \kappa(U_1)+\dots +\kappa(U_n)$ (subadditivity);

3) $\kappa(U_1\cup\dots\cup U_n) \le \max\{\kappa(U_1),\ldots,\kappa(U_n)\}$, if the sets $\cl U_1,\ldots, \cl U_n$ are pairwise disjoint.

We call such function $\kappa$ a \emph{generalized relative category}.
\end{defn}

The relative Lyusternik-Schnirelmann category $\cat_X U$ \cite{ls1934} is an example of a generalized relative category (when $X$ is arcwise connected to satisfy Property~3). 

Another example is the \emph{sectional category} of a map $f: Y\to X$, the definition from~\cite{schw1966} (see also~\cite{jam1978}) is as follows.

\begin{defn}
Let $f:Y\to X$ be a continuous surjective map. For an open subset $U\subseteq X$ we put 
$$
\secat_f U = 1
$$
if $f|_{f^{-1}(U)}$ has a section $U\to Y$. We put 
$$
\secat_f U = k
$$
if $k$ is the smallest size of an open cover $V_1\cup\dots\cup V_k\supseteq U$ by open sets with $\secat V_i=1$.
\end{defn}

The function $\secat_f U$ is a generalized relative category, Properties~1 and 2 are obvious, Property~3 holds because if $V_1,\ldots, V_n$ are open and pairwise disjoint with $\secat V_i=1$, then 
$$
\secat V_1\cup\dots\cup V_n=1,
$$
since the section of $f$ is defined on every $V_i$ separately.

A particular case of the sectional category is the Schwarz genus, introduced for $Z_2$-action by Krasnosel'skii, see~\cite{kras1955,schw1966}.

\begin{defn}
Let $G$ be a finite group, $Y$ be a free $G$-space. Define the \emph{Schwarz genus} $g_G(Y)$ as the sectional category of the natural projection $Y\to Y/G$.
\end{defn}

If the space $Y$ is fixed, and $\pi : Y\to Y/G=X$ is the natural projection, then $g_G(\pi^{-1}(U))$ is a generalized relative category on $X$. For paracompact $Y$ and its $G$-invariant open subset $U$ the property $g_G(U)\le n$ is equivalent to existence of a $G$-equivariant map $U\to \underbrace{G*G*\cdots*G}_{n} = G^{*n}$. In this case the genus also makes sense for closed invariant subsets. See Section~\ref{genus-sec} for further discussion of the Schwarz genus.

Let us define the cohomology length for a ring of coefficients $A$.

\begin{defn}
Let $X$ be a topological space, $\iota: U\to X$ the inclusion of an open subset. Define \emph{the relative cohomology length} of $U$ by
\begin{multline*}
\hl_X U = \max\{k : \exists u_1,\ldots, u_k\in H^*(X, A),\ \forall i\ \dim u_i >0,\\ \iota^*(u_1)\iota^*(u_2)\dots\iota^*(u_k)\not=0 \}.
\end{multline*}
\end{defn}

In Section~\ref{length-sec} we discuss the cohomology length and show that $\hl_X U +1$ is a generalized relative category. The subadditivity of length also implies the well-known bound $\cat_X U \ge \hl_X U + 1$. In fact, if a generalized relative category $\kappa$ takes contractible in $X$ subsets of $X$ to $1$, then $\kappa(U)\le \cat_X U$.

There are other examples of a generalized relative category for arcwise connected $X$. For any continuous map $f:X\to Y$ in~\cite{berga1961} the (restriction) category of a map is defined by the rule $\cat_f (U)=1$ iff $f|_U$ is null-homotopic. For any (generalized) cohomology class $\xi$ in~\cite{berga1961} the restriction category is defined by the rule $\cat_\xi U=1$ iff $\xi|_U=0$.

See the review~\cite{jam1978} for different generalizations of the Lyusternik-Schnirelmann category. Note that the \emph{strong category} of $U$, which requires the contractibility of $U$ in $U$ (not in $X$), is not a generalized relative category, the Property~3 obviously fails.

Now let us state the main result.

\begin{thm}
\label{fun-level} Let $\kappa$ be a generalized relative category on $X$ with $\kappa(X) > n(d+1)$. For any continuous map $f: X\to Y$ to a compact metric space of covering dimension $d$ there exists a point $c\in Y$ such that for any neighborhood $U\ni c$
$$
\kappa( f^{-1}(U)) > n.
$$
\end{thm}

This statement also holds if $Y$ is not compact, but $X$ is compact, in this case the image $f(X)$ is compact. 

If $\kappa$ can be defined for closed sets and homotopy invariant, and the preimage $f^{-1}(c)$ is a retract of its neightborhood in $X$, then the theorem claims that $\kappa(f^{-1}(c)) > n$. This is true for analytic maps from a real analytic variety $X$. 

If we consider the sectional category of a fiber bundle with ANR fibers (e.g. the Schwarz genus), then for closed $F\subseteq X$ for some neighborhood $U\supset F$ we have $\kappa(F)\ge \kappa(U)$. If, in addition, $X$ is compact, then the theorem claims that $\kappa(f^{-1}(c)) > n$.

Note that a particular case of Theorem~\ref{fun-level} for the $Z_2$-genus is proved and used in~\cite[Lemma~3.1]{yang1955}. A close result on a certain homological analogue of the genus is proved in~\cite[Lemma~9.1 and Corollary~9.1]{vol2005}. Actually, the idea of the proof for the cohomology length is contained in~\cite{vol2005}.

Another result in this direction if obtained in~\cite{vara1965} for fiber bundles, and the number $d+1$ is replaced by the category of the map. Let us state that result.

\begin{thm}
Let $F\xrightarrow{\iota} E \xrightarrow{p} B$ be a fiber bundle with connected $B$ and $F$. Then
$$
\cat E\le \cat \iota \cdot \cat p.
$$
\end{thm}

\section{Proof of the main theorem}

We need a lemma~\cite[Lemma~2.4]{pal1966A}.

\begin{lem}
\label{color-cover} Let $X$ be paracompact space of covering dimension $d$. Then any open covering $\mathcal V$ of $X$ has a refinement
$$
\mathcal U = \bigcup_{i=1}^{d+1} \mathcal U_i,
$$
where each $\mathcal U_i$ consists of the sets with pairwise disjoint closures.
\end{lem}

\begin{proof}[Proof of Theorem~\ref{fun-level}]
The space $Y$ is metric and therefore paracompact. By Lemma~\ref{color-cover} we can cover $Y$ by a family $\mathcal U = \mathcal U_1\cup\dots\cup \mathcal U_{d+1}$ of open sets with diameters at most $\varepsilon$. Put $V_i = \bigcup \mathcal U_i$.

By Property~2 of a generalized relative category for some $i$ we have $\kappa(f^{-1}(V_i)) > n$. By Property~3 there exists a set $V\in\mathcal U_i$, such that $\kappa(f^{-1}(V)) > n$.

If $\varepsilon$ tends to zero, we find such sets $V\subset Y$ with diameters tending to zero. By the compactness considerations, the sets $V$ tend to a point $c\in Y$. Let us show that $c$ is the required point. If there is a neighborhood $U\ni c$ such that $\kappa(f^{-1}(U)) \le n$, then for a fine enough covering we have $V\subseteq U$, which contradicts Property~1 and the fact that $\kappa(f^{-1}(V)) > n$.
\end{proof}

\section{Properties of the Schwarz genus}
\label{genus-sec}

Remind the porperties of the genus from~\cite{kras1955,yang1955II,schw1966}. A review on the genus and its generalizations for non-free actions can be found in~\cite{vol2005}. We assume the spaces to be paracompact.

\begin{lem}[Monotonicity]
If there exists a $G$-equivariant map $X\to Y$ between free $G$-spaces, then
$$
g_G(X) \le g_G(Y).
$$
\end{lem}

The following property obviously follows from $k-2$-connectivity of $G^{*k}$ and the obstruction theory.

\begin{lem}
\label{ind-dim}
For a paracompact $G$-space $X$ we have
$$
g_G(X)\le\dim X+1.
$$
\end{lem}

\begin{lem}
\label{ind-conn}
For an $(n-1)$-connected free $G$-space $X$ we have
$$
g_G(X) \ge n+1.
$$
\end{lem}

The Borsuk-Ulam theorem~\cite{bor1933} (in the version of Lusternik-Schnirelmann~\cite{ls1934}) follows from the above two lemmas:

\begin{lem}
\label{sphere-index} If $G$ acts freely on $S^d$, then
$$
g_G(S^d) = d+1.
$$
\end{lem}

Let us state another property of the genus, that is not listed in the properties of a generalized relative category.

\begin{thm}
\label{ls-gen} 
Let a free paracompact $G$-space $Y$ be covered by $G$-invariant open subsets $\{Y_i\}_{i=1}^l$. Then there exists $x\in Y$ such that
$$
\sum_{Y_i\ni x} g_G(Y_i) \ge g_G(Y).
$$
\end{thm}

Let us give an example of applying this theorem to describing the critical values of a smooth function. If $Y$ is a compact closed manifold with free action of $G$, $f$ is a smooth function on $X=Y/G$, $\xi$ is a gradient-like vector field for $f$, then either there are infinitely many critical points of $f$, or for every $\varepsilon>0$ there exists an integral curve of $\xi$, passing through $\varepsilon$-neighborhoods of at least $g_G(X)$ critical points of $f$.

\begin{proof}[Proof of theorem~\ref{ls-gen}]
Put $g_G(Y_i) = k_i$, we may assume these numbers to be finite, otherwise the statement trivially holds. Consider the equivariant maps $f_i : Y_i\to G^{*k_i}$, that exist by the definition of the genus (for paracompact spaces), and consider the $G$-invariant partition of unity $\{\rho_i\}_{i=1}^l$, subordinated to $\{Y_i\}$.

Define the map
$$
f : x \mapsto \rho_1(x)f_1(x)\oplus\dots\oplus\rho_l(x)f_l(x),
$$
which maps $X$ to the join $G^{*(k_1+\dots+k_l)}$. If the claim is not true, the image of this map is in the  $\left(g_G(X) - 2\right)$-dimensional skeleton of $G^{*(k_1+\dots+k_l)}$, which contradicts the monotonicity of the genus and Lemma~\ref{ind-dim}.
\end{proof}

\section{Properties of the cohomology length}
\label{length-sec}

The cohomology length is widely used to estimate the Lyusternik-Schnirelmann category. The following lemma is from~\cite{fe1935}, see also the review~\cite{jam1978}. Here $A$ is the ring of coefficients.

\begin{lem}
\label{hom-prod} Let a space $X$ be covered by open sets $U_1, U_2,\ldots, U_m$, and let the elements $a_1,
a_2,\ldots, a_m\in H^*(X, A)$ be given. If for any $i=1,\ldots, m$ the image of $a_i$ in $H^*(U_i, A)$ is zero, then the product $a_1a_2\cdots a_m = 0$ in $H^*(X, A)$.
\end{lem}

This lemma imply Property~2 of a generalized relative category for $\hl_X U+1$. Indeed, if a product $N=\sum_{i=1}^n \hl_X U_i + n$ of classes of positive dimension $u_1,\ldots,u_N\in H^*(X, A)$ is nonzero on $\bigcup_{i=1}^n U_i$, then this product can be partitioned into $n$ segments of length $\hl_X U_i + 1$ respectively. By Lemma~\ref{hom-prod} one of the segments is nonzero on its respective $U_i$, which contradicts the definition of $\hl_X U_i$.

The Properties~1 and 3 are obvious for the cohomology length.

\section{Some corollaries}
\label{corollaries-sec}

Let us give a corollary of Theorem~\ref{fun-level}. This corollary can also be deduced from the version of Theorem~\ref{fun-level} in~\cite{yang1955II}.

\begin{cor}[A generalized Borsuk-Ulam theorem for functions]
Let $k, l, n$ be positive integers such that $k(l+1) \le n$. Then for any $l$ continuous even functions $(f_1, \ldots, f_l)$ on $S^n$ there exist numbers $(c_1, \ldots, c_l)$ such that for any $k$ continuous odd functions $(g_1, \ldots, g_k)$ the system of equations
\begin{eqnarray*}
f_1(x) = c_1,\ f_2(x) = c_2,\ldots, f_l(x) = c_l\\
g_1(x) = 0,\ g_2(x) = 0,\ldots, g_k(x) = 0
\end{eqnarray*}
has a solution.
\end{cor}

\begin{proof}
By Lemma~\ref{sphere-index}, $Z_2$-genus of $S^n$ under the involution $x\mapsto -x$ equals $n+1$. Let us apply Theorem~\ref{fun-level} to the map $f : S^n/Z_2\to \mathbb R^l$, given by the functions $(f_1, \ldots, f_l)$. We obtain a set of numbers $(c_1, \ldots, c_l)$ such that the subset of $S^n$ (see the comment after Theorem~\ref{fun-level})
$$
Y = \{x\in S^n : f_1(x) = c_1,\ f_2(x) = c_2,\ldots, f_l(x) = c_l\}
$$
has genus at least $k$. Similarly to the proof of the original Borsuk-Ulam theorem, Lemma~\ref{sphere-index} and the monotonicity of the genus imply, that any set of continuous odd functions $(g_1, \ldots, g_k)$ has a common zero on $Y$.
\end{proof}

If we take the functions $(g_1, \ldots, g_k)$ to be linear, we obtain the following result.

\begin{cor}
Let $k, l, n$ be positive integers such that $k(l+1) \le n$. Then for any $l$ continuous even functions $(f_1, \ldots, f_l)$ on $S^n$ there exist numbers $(c_1, \ldots, c_l)$ such that the set
$$
Y = \{x \in S^n : f_1(x) = c_1,\ f_2(x) = c_2,\ldots, f_l(x) = c_l\}
$$
intersects with any linear subspace of dimension $n+1-k$ in $\mathbb R^{n+1}$.
\end{cor}

\end{document}